\documentclass[11pt]{amsart}

\usepackage{amsmath,amsthm,amssymb,amsfonts, amsopn, mathtools}
\usepackage{stmaryrd,tikz,tikz-cd}
 \usepackage{todonotes}
\usepackage{tcolorbox}
\usepackage[colorlinks=true, linktocpage=true]{hyperref}
\hypersetup{citecolor=blue}
\usepackage[all]{xy}

\usepackage{stackengine}
\def\subrangle#1{\stackengine{5pt}{}{$\!\scriptstyle #1$}{U}{l}{F}{F}{L}}
\newcommand{\norm}[1]{\left\lVert#1\right\rVert}

\newtheorem{theorem}{Theorem}[section]
\newtheorem{corollary}[theorem]{Corollary}
\newtheorem{lemma}[theorem]{Lemma}         
\newtheorem*{lemma*}{Lemma}         
\newtheorem{proposition}[theorem]{Proposition}
\newtheorem*{proposition*}{Proposition}

\theoremstyle{definition}
\newtheorem{definition}[theorem]{Definition}    
\newtheorem{remark}[theorem]{Remark}          

\numberwithin{equation}{section}

\newcommand{\C}{\mathbb C}
\newcommand{\A}{\mathbb A}

\newcommand{\oo}{\mathcal O}

\newcommand{\End}{\mathrm{End}}

\title{Adelic $C^{*}$-correspondences and parabolic induction}

\author[Goffeng]{Magnus Goffeng}
\address{\normalfont{Centre for Mathematical Sciences, Lund University \\
Box 118, SE-221 00 Lund, Sweden}}
\email{magnus.goffeng@math.lth.se}

\author[Mesland]{Bram Mesland}
\address{\normalfont{Mathematisch Instituut, Universiteit Leiden \\
Postbus 9512, 2300 RA Leiden, Netherlands}}
\email{b.mesland@math.leidenuniv.nl}

\author[\c{S}eng\"un]{Mehmet Haluk \c{S}eng\"un}
\address{\normalfont{School of Mathematical and Physical Sciences\\
University of Sheffield,
Hounsfield Road, Sheffield, S3 7RH, UK}}
\email{m.sengun@sheffield.ac.uk}

\begin{document}

\begin{abstract} 
In analogy with the factorization of representations of adelic groups as restricted products of representations of local groups, we study restricted tensor products of Hilbert $C^*$-modules and of $C^*$-correspondences. The construction produces global $C^*$-correspondences from compatible collections of local $C^{*}$-correspondences. When applied to the collection of $C^{*}$-correspondences capturing local parabolic induction, the construction produces a global $C^*$-correspondence that captures adelic parabolic induction.
\end{abstract}

 \maketitle
\tableofcontents
\section{Introduction}
Historically, $C^*$-algebras and representation theory share a common ancestry \cite{mvn, Segal-47, Segal-50, harcha, glimm61}. In the 1960's, representation theory of reductive groups over rings of adeles of number fields has started to acquire a central position in number theory and automorphic forms theory, and this motivated some $C^*$-algebraists to consider restricted tensor products of $C^*$-algebras \cite{Blackadar, Guichardet}. Recent years have seen several striking applications of the notion of $C^*$-correspondences from the theory of $C^*$-algebras to the theory of representations of reductive groups over local fields \cite{Clare-13, CCH-16, CCH-18, MS-24}. Motivated by these developments, we explore the notion of restricted tensor product in the setting of $C^*$-correspondences (and, as a first step, of Hilbert modules) with applications to representation theory of adelic reductive groups.

The main results of this paper are of a technical nature showing how to form the restricted tensor product of $C^*$-correspondences, followed by an example placing parabolic induction for adelic reductive groups in the context of $C^*$-correspondences. 
In more detail, the technical motivation for the paper is the following question. Consider two collections of $C^*$-algebras $(A_v)_{v\in I}$ and $(A'_v)_{v\in I}$ and a collection of $(A_v,A_v')$-correspondences $(X_v)_{v\in I}$. How do we patch these collections together in an infinite restricted tensor product? We answer this question in this paper by enhancing the $C^{*}$-algebras with appropriate projections $p_v\in A_v$, $p_v'\in A_v'$ and the modules with distinguished vectors $x_v\in X_v$, satisfying the conditions
$$\langle x_v, x_v \rangle_v = p_v, \qquad  p'_v{\cdot}x_v = x_v$$
for all but finitely many indices $v$.
With these auxillary data and the conditions at hand, we construct the restricted tensor product $\bigotimes_{v\in I}^{'}(X_{v},x_{v})$ as a $C^{*}$-correspondence for the pair of restricted tensor product $C^{*}$-algebras $(\bigotimes_{v\in I}'(A'_{v},p_{v}'),\bigotimes_{v\in I}'(A_{v},p_{v}))$.  Without the extra data of projections and distinguished vectors, a restricted tensor product would not make sense. 

To illustrate the utility of our construction, we turn to representation theory which was our main motivation as mentioned earlier. In \cite{Pierrot-01, Clare-13},  local parabolic induction functor has been interpreted as $C^*$-correspondence between suitable group $C^*$-algebras. This approach has been further studied and utilized in \cite{CCH-16,CCH-18}.  We show here that the adelic parabolic induction functor can be captured as the restricted tensor product of the local parabolic induction $C^*$-correspondences.  This can be seen as a $C^*$-algebraic counterpart of the local-global compatibility of parabolic induction. A further application of our results to global theta correspondence is pursued in the paper \cite{GMSTheta}.

The paper is organized as follows. In Section \ref{sec:back} we recall the relevant background on restricted tensor products of vector spaces, Hilbert spaces and $C^*$-algebras. Further details for adelic groups are discussed in Section \ref{sec:adelic}, with a particular focus on their representation theory and $C^*$-algebras. The main construction can be found in Section \ref{main-results}, where we construct the restricted tensor product of a collection of Hilbert $C^*$-modules with a distinguished vector chosen. In the cases of interest in representation theory, one considers restricted tensor products of a collection of not just of Hilbert $C^*$-modules but of $C^*$-correspondences and considerations thereof can be found in Section \ref{main-resultscorr}. We discuss applications to induction for adelic groups in Section \ref{sec:globalpara}.\\

{\bf Acknowledgements.} We are grateful to the referee whose detailed and insightful report provided many simplifications and strengthenings, especially in Section  \ref{sec:globalpara}.
\section{Background}
\label{sec:back}

The key technical tool in constructing the adeles as well as groups thereover comes from restricted products or tensor products. We here recall the incarnation of restricted (tensor) products in various classes of objects. The history of infinite tensor products can be traced back to von Neumann \cite{vonNeumann}. A good source for the materials below is \cite{Guichardet}. See also \cite{Blackadar}.

\subsection{Vector spaces} 
\label{vector-space} 

Let $(W_v)_{v \in I}$ be a family of vector spaces indexed by a countable index set $I$. Let $I_0$ be a fixed finite subset of $I$. For all $v \not \in I_0$, fix a vector $x_v \in W_v$. 
For a finite subset $S\subseteq I$ containing $I_0$, put $W_S := \bigotimes_{v \in S} W_v.$ Then for $S'=S \sqcup \{ v' \}$, we can consider the linear map
\begin{equation} \label{vs-embedding} W_S \rightarrow W_{S'}=W_S \otimes W_{v'}, \quad w \mapsto w \otimes x_{v'},
\end{equation}
which is an embedding if $x_{v'}$ is nonzero. 
Using these maps, we turn the collection $W_S$ with $S$ running over finite subsets of $I$ containing $I_0$ into a directed system and form
$$\sideset{}{'}\bigotimes_{v\in I} (W_v,x_v) := \varinjlim_S W_{S}=\bigcup_S W_S.$$

Note that $\sideset{}{'}\bigotimes_{v\in I} (W_v,x_v)$ is spanned by elements $w=\otimes_v w_v$ such that for all but finitely many $v$, we have $w_v$ is the distinguished vector $x_v$.

If $B_v : W_v \to W_v$ are linear maps such that $B_v(x_v)=x_v$ for all but finitely many $v$, we have a linear operator
$$\sideset{}{'}\bigotimes_{v\in I} B_v :  \sideset{}{'}\bigotimes_{v\in I} (W_v,x_v)  \longrightarrow \sideset{}{'}\bigotimes_{v\in I} (W_v,x_v), \qquad \otimes_{v} w_v \mapsto \otimes_v B_v(w_v).$$
Note that if infinitely many $x_v$are zero, then $\sideset{}{'}\bigotimes_{v\in I} (W_v,x_v)=0$. We will not consider this case in this paper as it will not arise in the applications of Section \ref{sec:globalpara}, however, 
it will arise in the companion paper \cite{GMSTheta} where we consider the global $\theta$-correspondence.

\subsection{Hilbert spaces} 
\label{hilbert-space} 

When the vector spaces in the preceding section are Hilbert spaces $\mathcal{H}_v=W_v$, then the restricted tensor product can be given a Hilbert space structure provided that the distinguished vectors $h_v=x_v$ are in the unit ball, and  $\|h_{v}\|=1$ for all but finitely of them. 
Then the maps (\ref{vs-embedding}) are isometries for all but finitely many indices  and the direct limit $\varinjlim_{S} \mathcal{H}_{S}=\bigcup_S \mathcal{H}_S$ carries a canonical pre-Hilbert space structure. We denote its completion as 
$$\sideset{}{'}\bigotimes_{v\in I} (\mathcal{H}_v,h_v).$$

Note that for simple tensors $e=\otimes_{v}e_v,f=\otimes_{v}f_v \in \bigcup_S \mathcal{H}_S$, we have
$$\langle e,f \rangle = \prod_{v} \langle e_v, f_v \rangle\subrangle{v}$$
which is well defined because for all but finitely many indices, we have $e_v=f_v=h_v$ so that 
$$ \langle e_v, f_v \rangle\subrangle{v} =  \langle h_v, h_v \rangle\subrangle{v}=1$$
by our assumption that $h_v$ are unit vectors.

If $B_v : \mathcal{H}_v \to \mathcal{H}_v'$ are bounded linear maps between two collections of Hilbert spaces, with distinguished unit vectors $(h_v)_{v \not \in I_0}$ and $(h_v')_{v \not \in I_0}$, such that 
\begin{enumerate}
\item $B_v(h_v)=h_v'$ for all but finitely many $v$, 
\item $\norm{B_v} \leq 1$ for all but finitely many $v$, 
\end{enumerate}
then the linear map
\begin{equation}
\label{resttensofops}
\sideset{}{'}\bigotimes_{v\in I} B_v :  \sideset{}{'}\bigotimes_{v\in I} (\mathcal{H}_v,h_v) \longrightarrow \sideset{}{'}\bigotimes_{v\in I} (\mathcal{H}'_v,h'_v), \qquad \otimes_v w_v \mapsto \otimes_v B_v(w_v),
\end{equation}
is bounded as well. Note that the two assumptions above in fact imply that $\norm{B_v} = 1$ for all but finitely many $v$, and so 
\begin{equation}
\label{prodnorm}
\left\|\sideset{}{'}\bigotimes_{v\in I} B_v\right\|=\prod_{v\in I} \|B_v\|,
\end{equation}
where the factors in the product are $1$ for all but finitely many $v$. It follows from Proposition \ref{comaokdapdops} below that all compact operators $\sideset{}{'}\bigotimes (\mathcal{H}_v,h_v) \longrightarrow \sideset{}{'}\bigotimes (\mathcal{H}'_v,h'_v)$ can be approximated in norm by linear combinations of operators of the form \eqref{resttensofops} (for $B_v$ compact). The analogous statement fails for bounded operators.

\subsection{$C^*$-algebras} 
\label{basic-family} 

Let $(A_v)_{v \in I}$ be a collection of $C^*$-algebras indexed by a countable set $I$. The $C^*$-algebras $A_v$'s will in our applications be type I and  nuclear, but in this section we state explicitly when such assumptions are used. We shall use $\otimes$ to denote the spatial tensor product of $C^{*}$-algebras. Assume that for all $v$ not in a finite subset $I_0$, we have a distinguished nonzero projection $p_v \in A_v$ that we fix. By abuse of notation, we will denote such a family simply by 
$$(A_v,p_v)_{v \in I}.$$

Given a finite subset $S$ of $I$ containing $I_0$, we define the $C^*$-algebra 
$$A_S:=\bigotimes_{v \in S} A_v.$$
Then for $S'= S \sqcup \{ v' \}$, we form the $*$-homomorphism
\begin{equation}  
\label{aftoafprime}
A_S \hookrightarrow A_{S'}= A_S \otimes A_{v'}, \quad a \mapsto a \otimes p_{v'},
\end{equation}
which is injective, hence isometric.
These turn the collection $(A_S)_S$, for $S$ ranging over finite subsets of $I$ containing $I_0$
into a directed system of $C^*$-algebras. We define the restricted product of the collection $(A_v, p_v)_{v \in I}$ as the direct limit of this directed system
$$\sideset{}{'}{\bigotimes}_{v \in I} (A_v,p_v):=\varinjlim_S A_S.$$

 The direct limit is taken in the category of $C^*$-algebras and we can describe the $C^*$-algebra structure explicitly. The direct limit $\varinjlim_S A_S$ is defined as the closure of $\bigcup_{S} A_S$ in the $C^*$-norm defined from each $A_S$. This is well defined since for such $S,S'$ the map in \eqref{aftoafprime} is isometric.

\subsection{Representations} 

Let $(A_v,p_v)_{v \in I}$ be as in Section \ref{basic-family}. Given $v$, let us fix a representation $(\pi_v, \mathcal{H}_{v})$ of $A_v$, that is, $*$-homomorphisms
$$\pi_v : A_v \to \mathbb{B}(\mathcal{H}_{v})$$
for Hilbert spaces $\mathcal{H}_{v}$. 

If for all $v \not \in I_0$, we fix unit vectors $h_v \in \mathcal{H}_{v}$ such that $\pi_v(p_v)(h_v)=h_v$ then we can, using the norm  computation \eqref{prodnorm}, form a representation 
\begin{equation}
\label{defofpi}
\pi:= \sideset{}{'}\bigotimes_{v\in I} \pi_v \ : \sideset{}{'}\bigotimes_{v\in I} (A_v,p_v) \longrightarrow \mathbb{B}\left (\sideset{}{'}\bigotimes_{v\in I} (\mathcal{H}_{v},h_v) \right ),
\end{equation}
as follows: given $a=\otimes_{v}a_v \in \sideset{}{'}\bigotimes (A_v,p_v)$ and $w=\otimes_{v}w_v \in \sideset{}{'}\bigotimes (\mathcal{H}_{v},h_v) $, we define $\pi(a)(w)$ via the rule
$$\pi(a)(w) := \otimes_{v}\pi_v(a_v)(w_v).$$
At all but finitely many indices $v$, we have $a_v=p_v$ and $w_v=h_v$, and thus $\pi_v(a_v)(w_v)=w_v$. It follows that $\pi(a)(w)$ belongs to $\sideset{}{'}\bigotimes (\mathcal{H}_{v},h_v)$, and from the discussion in Section \ref{hilbert-space} we deduce that $\pi(a)$ extends by linearity and continuity to a bounded linear operator. This construction extends by linearity and continuity to all of $\sideset{}{'}\bigotimes_{v\in I} (A_v,p_v)$ as claimed in \eqref{defofpi}.

The representation $\pi$ is irreducible if all the $\pi_v$ are irreducible. Modifying the fixed vectors $h_v$ at finitely many indices $v$ does not change the isomorphism class of $\pi$. The same is true if we replace $h_v$'s with a scalar multiple. Thus, if the projections $\pi_v(p_v)$ all have rank one for all but finitely many $v$, then $\pi$ does not depend on the choice of $h_v$. 

\begin{definition}[Definition 12 in \cite{Guichardet}] A projection $p$ in a $C^*$-algebra $A$ has {\em rank at most one} if for every irreducible representation $(\pi, \mathcal{H})$ of $A$, the projection $\pi(p)$ has rank at most one. 
\end{definition}
The next result is due to Guichardet (see \cite[Section 13]{Guichardet}).

\begin{theorem} 
\label{factorization-Cstar} 
Let $(A_v,p_v)_{v \in I}$ be a collection of Type I $C^*$-algebras with all but finitely many of the distinguished projections $p_v$ having rank at most one. Every irreducible representation of $\sideset{}{'}\bigotimes (A_v,p_v)$ is equivalent to a restricted tensor product as in \eqref{defofpi} of irreducible representations $\pi_v$ of $A_v$, $v\in I$.
\end{theorem}

The reader should note that it is crucial that we use the spatial tensor product in Theorem \ref{factorization-Cstar}. For the maximal tensor product there are non-factorizable irreducible representations.

\section{Adelic groups}
\label{sec:adelic}

We now turn to discuss how restricted products of groups produce adelic groups and describe their representation theory. Some useful sources are \cite{Flath}, \cite[Section 3]{Gelfand_etal} and \cite[Sections 5,6]{Moore}.

\subsection{Restricted products of groups} 
\label{setup} 
Let $(G_v)_{v \in I}$ be a countable collection of locally compact, second countable, Hausdorff groups. We assume that each $G_v$ is a Type I group, that is, its $C^*$-algebra $C^*(G_v)$ is of Type I. Since each group $G_v$ is Type I, each of the $C^*$-algebras $C^*(G_v)$ are nuclear by \cite{patte}. Let us fix subgroups $K_{v}\subset G_{v}$ for all $v$ outside a finite subset $I_0$.  By abuse of notation, we will denote such a collection simply by
$$(G_v,K_{v})_{v\in I}.$$

\begin{definition}
\label{def:adm} 
A collection $(G_v,K_{v})_{v\in I}$ of groups as above is called an \emph{admissible family} if each $v \not \in I_0$, the subgroup $K_{v}$ is  {\em compact open} in $G_v$. 
\end{definition}
The {\em restricted product} of an admissible family $(G_v,K_v)_{v \in I}$ is defined as
\begin{eqnarray*}
\sideset{}{'}\prod_{v\in I} (G_v : K_v) & :=& \left \{ (g_v)_v \in {\prod}_{v \in I} G_v \mid g_v \in K_v \ \ \textrm{for all but finitely many} \ v \right \} \\ 
& =& \bigcup_{S} \  \prod_{v \in S} G_v \times \prod_{v \not \in S} K_{v}
\end{eqnarray*}
where $S$ varies over finite subsets of $I$ containing $I_0$. Equipping with the direct limit topology, we obtain a locally compact, Hausdorff topological group. 

In the rest of the paper, we will put
\begin{equation} \label{G_S} G_S:=\prod_{v \in S} G_v \times \prod_{v \not \in S} K_{v}
\end{equation}
so that 
$$\sideset{}{'}\prod_{v\in I} (G_v : K_v) =  \bigcup_{S} G_S.$$
Note that each $G_S$ is an open subgroup of $\sideset{}{'}\prod (G_v : K_v)$.

\subsection{Representations of restricted products of groups} 

Let $(G_v,K_v)_{v \in I}$ be an admissible family of groups as in Definition \ref{def:adm}. Let unitary representations
$$\pi_v : G_v \to {\mathcal U}(\mathcal{H}_v)$$
be given. Assume that for all $v \not \in I_0$, there exists a unit vector $h_v \in \mathcal{H}_v$ such that 
\begin{center} $h_v$ is fixed by $K_v$.
\end{center}
Using the previous constructions, we can put the $\pi_v$'s together
$$\pi := \sideset{}{'}\bigotimes \pi_v.$$
This is an irreducible unitary representation of $\sideset{}{'}\prod (G_v : K_v)$ afforded on the Hilbert space $\sideset{}{'}\bigotimes (\mathcal{H}_v,h_v)$. 

In order to describe the dependence of $\pi$ on the distinguished vectors, we recall the notion of \emph{Gelfand pair}. Let $G$ be a topological group, $K\subset G$ a compact subgroup and $\mathcal{H}(G,K)$ the Hecke algebra, that is, the convolution algebra of continuous, compactly supported bi-$K$-invariant complex valued functions on $G$. Then $(G,K)$ is called a Gelfand pair if $\mathcal{H}(G,K)$ is commutative.

Two collections $(h_v)_{v \not \in I_0}$, $(h'_v)_{v \not \in I_0}$ of distinguished vectors give rise to isomorphic representations if $h_v$ is a scalar multiple of $h'_v$ for all but finitely many $v$. Thus the following proposition (see e.g. \cite[Proposition 6.3.1]{vanDijk_book}) is relevant.

\begin{proposition} 
\label{aodnaodjna}
Let $(G,K)$ be a Gelfand pair. For any irreducible unitary $G$-module $V$, the subspace $V^{K}$ of $K$-fixed vectors is at most one-dimensional. 
\end{proposition}

Thus if all but finitely many $(G_v,K_v)$ form a Gelfand pair, then the restricted product representation $\pi$ will be independent of the chosen collection of distinguished vectors. 

In the converse direction, we have the following factorization result. 

\begin{theorem} 
\label{factorization} 
Let $(G_v,K_v)_{v \in I}$ be an admissible family of groups such that  $(G_{v},K_{v})$ is a Gelfand pair for all but finitely many $v$. Every irreducible unitary representation 
$\pi$ of the restricted product group $\sideset{}{'}{\prod} (G_v : K_v)$ is factorizable into local unitary irreducible representations: 
$$\pi \simeq \sideset{}{'}\bigotimes \pi_v.$$ 
The isomorphism classes of the unitary representations $(\pi_v, \mathcal{H}_v)$ are determined by that of $\pi$. For all but finitely many $v$, the dimension of $\mathcal{H}_v^{K_v}$ is one.
\end{theorem}

We note that Theorem \ref{factorization} follows from Guichardet's Theorem \ref{factorization-Cstar}. There is an analogous factorization theorem of Flath \cite{Flath} (see also \cite[Section 5.7]{getzhahn}) for ``admissible'' irreducible representations of a reductive group over the adeles. Such representations form a larger class that contain the unitary ones and accordingly one has to apply the strategy of the proof of Guichardet's Theorem to the Hecke algebra of the reductive group in order to achieve the factorization.

\subsection{Adeles and groups over adeles} \label{adelic_groups} Let $F$ be a number field. If $v$ is a place of $F$, let $F_v$ denote the completion of $F$ at $v$. All $F_v$ are locally compact. Let $I_0$ denote the set of infinite places of $F$. For all finite places, that is $v \not \in I_0$, let $\oo_v$ denote the unique maximal compact open subring of $F_v$. The restricted product of (additive) groups
\[\A=\A_F:=\sideset{}{'}\prod(F_{v}:\mathcal{O}_{v})\equiv \prod_{v\in I_0} F_{v}\times \sideset{}{'}\prod_{v\notin I_0}(F_{v}:\mathcal{O}_{v}),\]
can be made into a locally compact ring called the {\em ring of adeles} of $F$. 

Let $G$ be a linear\footnote{That is, $G$ embeds into ${\rm GL}_n$ for some $n$.} algebraic group defined over a number field $F$. For convenience, let us put 
$$G_v:=G(F_v)$$
for a place $v$ of $F$.  For all $v \not \in I_0$, we put 
$$K_v:=G(\oo_v):= G(F_v) \cap {\rm GL}_n(\oo_v)$$
Then $K_v$ is a {\em compact open} subgroup of $G(F_v)$. 

The {\em adelic group} $G(\A)$ obtained by taking the adelic points of $G$ is isomorphic as a topological group to the restricted product of $G_v$ with respect to $K_v$:
$$G(\A) \simeq \sideset{}{'}\prod (G_v : K_v).$$

\begin{proposition} 
\label{reductive-GelfandPair} 
If $G$ is reductive, then for all but finitely many of the finite places $v$, $(G_v,K_v)$ is a Gelfand pair. Thus, irreducible unitary representations of $G(\A)$ are factorizable in the sense of Theorem \ref{factorization}. 
\end{proposition}

See \cite{Tits} for a proof. We mention in passing that when $G$ is reductive, $G(\A)$ is a Type I group. This is not necessarily true if $G$ is not reductive (see \cite{Moore}).

\subsection{$C^*$-algebras of restricted products of groups}
\label{sec:calggroup}
For a group $G$, we write $\epsilon_G$ for the trivial $G$-representation and if $G$ is compact we tacitly identify the trivial representation with its support projection in $C^*(G)$. We note that for a compact group $K$, the support projection $\epsilon_K$  is represented by the constant function $1_K\in C(K)\subseteq C^*(K)$ (assuming the volume is normalized to one).

\begin{proposition} 
\label{propdecomomd}
Let $(G_v,K_v)_{v \in I}$ be an admissible family of groups as in Definition \ref{def:adm} and $S\subseteq I$ a finite subset  containing $I_0$.
Let $G_S$ be as in (\ref{G_S}). We have that 
$$C^*(G_S)\simeq  \bigotimes_{v \in S} C^*(G_v) \otimes C^*(K_S),$$
and 
$$C^*_r(G_S)\simeq  \bigotimes_{v \in S} C^*_r(G_v) \otimes C^*(K_S),$$
and the $C^*$-algebra of $K_S$ is the restricted tensor product
$$C^*(K_S)\simeq C^*_r(K_S)\simeq \sideset{}{'}{\bigotimes}_{v\notin S} (C^*(K_v),\epsilon_{K_v}).$$
\end{proposition}

\begin{proof}
The first two isomorphisms follow from the corresponding product decomposition at the level of groups. A priori, the product decomposition for the first isomorphism gives that $C^*(G_S)$ is the maximal tensor product $\bigotimes_{v \in S}^{\rm max} C^*(G_v) \otimes C^*(K_S)$ but each $G_v$ is type $I$ so $C^*(G_v)$ is nuclear by \cite{patte} and $K_S$ is compact so the full and reduced $C^*$-algebras of $K_S$ coincide and are nuclear

We define $A_0\subseteq C(K_S)$ as the subalgebra of continuous functions that are constant except in finitely many places, i.e. functions $f\in C(K_S)$ of the form $f((k_v)_{v\notin S})=f_0((k_v)_{v\in S'\setminus S})$ for a finite super set $S'\supseteq S$. Since the support projection of the trivial representation of a compact group is represented by the constant function (assuming the volume is normalized to one), it is clear that $A_0$ naturally embeds as a dense $*$-subalgebra of both $C^*(K_S)$ and $\bigotimes_{v\notin S} (C^*(K_v),\epsilon_{K_v})$ so the isomorphism exists by the universal property of group $C^*$-algebras.
\end{proof}

Proposition \ref{propdecomomd} can also be found in the literature as \cite[Example (3), page 316]{Blackadar}.

\begin{lemma} 
\label{bmplem} (Baum-Millington-Plymen) 
Let $G$ be a group that is the union of ascending chain of open subgroups $G_n$. Then the inclusions
$$\bigcup_{n} C^*(G_n)\subseteq C^*(G) \quad \mbox{and}\quad \bigcup_{n} C^*_r(G_n)\subseteq C^*_r(G)$$
are dense, or in other words $C^*(G)=\varinjlim C^*(G_n)$ and $C^*_r(G)=\varinjlim C^*_r(G_n)$ in the category of $C^*$-algebras.
\end{lemma}

In fact, this lemma is proven for the reduced group $C^*$-algebra in \cite{Baum_etal} but the proof works for the maximal group $C^*$-algebra as well.

\begin{lemma} 
\label{factorization-grouprep}
Let $(G_v,K_v)_{v \in I}$ be an admissible collection of groups as in Definition \ref{def:adm}. We have 
$$C^*\left(\sideset{}{'}{\prod} (G_v : K_v)\right) \simeq  \sideset{}{'}\bigotimes_v (C^*(G_v) , p_{K_v}),$$
and 
$$C^*_r\left(\sideset{}{'}{\prod} (G_v : K_v)\right) \simeq  \sideset{}{'}\bigotimes_v (C^*_r(G_v) , p_{K_v}).$$
Here, $p_{K_v}$ is defined for all $v \not \in I_0$, as the projection in $C^*(G_v)$, or $C^*_r(G_v)$, given by the characteristic function of the compact open subgroup $K_v$ whose volume is normalized to be 1. 
\end{lemma}

For compact open subgroups $K_v$, we have that $C^*(K_v)$ is a subalgebra of $C^*(G_v)$. Recall that when viewed as an element of the subalgebra $C^*(K_v)$, the projection $p_{K_v}$ is simply the constant function 1, that is, $\epsilon_{K_v}$ that we used earlier.

\begin{proof}
We only provide a proof for the maximal $C^*$-algebras, the statement in the reduced setting goes ad verbatim. Recall that $\sideset{}{'}{\prod} (G_v : K_v)=\bigcup_{S} G_S$ (cf. \eqref{G_S}) with $S$ running over finite subsets of $I$ containing $I_0$ and that each $G_S$ is open. From Lemma \ref{bmplem}, we conclude that we have a dense inclusion
$$\bigcup_{S} C^*(G_S)\subseteq C^*\left(\sideset{}{'}{\prod} (G_v : K_v)\right).$$
The inclusion is isometric on each $C^*(G_S)$. From Proposition \ref{propdecomomd}, we conclude that 
$$ \overline{\bigcup_S \left ( \bigotimes_{v \in S} C^*(G_v) \otimes  \bigotimes_{v\notin S} (C^*(K_v),p_{K_v}) \right )}\subseteq C^*\left(\sideset{}{'}{\prod} (G_v : K_v)\right),$$
is dense. At all $v \not \in I_0$, the subgroup $K_v$ is compact open in $G_v$ and therefore $p_{K_v}\in C^*(K_v)\subseteq C^*(G_v)$. It follows that
\begin{equation}
\label{jnadnadjn}
\bigcup_S \left (  \bigotimes_{v \in S} C^*(G_v) \otimes  \bigotimes_{v\notin S} (C^*(K_v),p_{K_v}) \right )\subseteq  \sideset{}{'}\bigotimes_v (C^*(G_v) , p_{K_v}),
\end{equation}
is dense. Therefore $\bigcup_S \left (  \bigotimes_{v \in S} C^*(G_v) \otimes  \bigotimes_{v\notin S} (C^*(K_v),p_{K_v}) \right )$ canonically embeds isometrically as a dense subalgebra of both $C^*\left(\sideset{}{'}{\prod} (G_v : K_v)\right)$ and $ \sideset{}{'}\bigotimes_v (C^*(G_v) , p_{K_v})$, proving the lemma.
\end{proof}

\begin{corollary} 
\label{equitotad}
Let $G$ be a linear algebraic group over a number field $F$. Then 
$$C^*\left(G(\A) \right) \simeq  \sideset{}{'}\bigotimes_v (C^*(G_v) , p_{K_v}), \quad \mbox{and}\quad C^*_r\left(G(\A) \right) \simeq  \sideset{}{'}\bigotimes_v (C^*_r(G_v) , p_{K_v}).$$
Here, the $p_{K_v}$ are defined for all $v \not \in I_0$ as in Lemma \ref{factorization-grouprep}.
\end{corollary}

\begin{remark} 
\label{adjnakdjnadkjn}
Corollary \ref{equitotad} implies the result of Tadi\'c \cite{Tadic} which describes the the unitary dual of $G(\A)$ as a restricted product in terms of those of $G_v$'s. 
\end{remark}

\section{Restricted tensor products of Hilbert $C^*$-modules}  
\label{main-results}

We now turn our attention to defining restricted tensor products of Hilbert $C^*$-modules. For Hilbert $C^*$-modules, restricted tensor products will be defined in a similar spirit as for Hilbert spaces (see Subsection \ref{hilbert-space}) but more care is needed for the fixed vectors. 

Recall that a Hilbert $C^*$-module $X$ over a $C^*$-algebra $A$ is a right $A$-module $X$ equipped with an $A$-valued inner product $\langle {\cdot},{\cdot}\rangle\subrangle{A}$ making $X$ into a Banach space in the norm 
\[\|x\|_X:=\|\langle x,x\rangle\subrangle{A}\|^{\tfrac{1}{2}}.\] 
The $C^*$-algebra of adjointable $A$-linear operators on $X$ is denoted by $\End_A^*(X)$. An element $T\in \End_A^*(X)$ is said to be $A$-compact if it is a norm limit of sums of rank one module operators
\begin{equation} \label{rank-one} T_{\xi,\eta}: x\mapsto \xi\langle \eta,x\rangle_X,
\end{equation}
for $\xi,\eta\in X$. We write $\mathbb{K}_A(X)\subseteq \End_A^*(X)$ for the ideal of $A$-compact operators. For more details on Hilbert $C^*$-modules, see \cite{lancebook}.

\subsection{Compatible collections}
Let $(A_v,p_v)_{v \in I}$ be a collection of $C^*$-algebras equipped with distinguished projections as in Subsection \ref{basic-family}. We also consider a collection $(X_v)_{v \in I}$ of Hilbert $C^*$-modules $X_v$ over $A_v$. We write $\langle {\cdot},{\cdot}\rangle\subrangle{v}$ for the $A_v$-valued inner product on $X_v$.

\begin{definition} 
\label{def:compvector}
Assume that we are given a distinguished vector $x_v\in X_v$ for all $v \not \in I_0$.  We say that  $(X_{v},x_{v})_{v\in I}$ is a \emph{compatible collection} of right Hilbert $C^{*}$-modules over $(A_{v},p_{v})_{v\in I}$, if for all but finitely many $v$, we have  
$$\langle x_v, x_v \rangle\subrangle{v}=p_v.$$ 
\end{definition}

\begin{remark}
\label{rmk:automatic}
Note that compatibility implies that for all but finitely many $v$, we have $x_{v}=x_{v}p_{v}$, as $\langle x_{v}-x_{v}p_{v},x_{v}-x_{v}p_{v}\rangle_{v}=0$. Also, the rank one operators (see (\ref{rank-one}))
\begin{equation} \label{pxv} p_{x_{v}}:=T_{x_{v},x_{v}}\in\mathbb{K}(X_{v})
\end{equation} 
are projections.
\end{remark}

We put
$$A:=\sideset{}{'}{\bigotimes}_{v \in I} (A_v,p_v) \quad\mbox{and}\quad A_{0}:=\bigcup_S A_S$$
where $S$ ranges over finite subsets of $I$ containing $I_0$. 
Note that $A_{0}\subseteq A$ is a dense $*$-subalgebra.

We shall now show that a compatible collection $(X_v,x_{v})_{v \in I}$ of Hilbert $C^*$-modules over $(A_v,p_{v})_{v \in I}$ admits a restricted tensor product, which is a Hilbert $C^*$-module over the restricted tensor product $C^{*}$-algebra $A$ of $(A_v,p_{v})_{v \in I}$. 

Following Subsection \ref{vector-space}, for any $S$ as above, we define the right $A_S$-Hilbert $C^*$-module $X_S$ as the exterior tensor product 
\begin{equation} 
\label{X_S-define} 
X_S := \bigotimes_{v \in S} (X_v,x_v).
\end{equation}
Indeed, $X_S$ is a Hilbert $C^*$-module over $A_{S}$ because it is a finite exterior tensor product. Then for $S'=S \sqcup \{ v' \}$, we will consider the map
\begin{equation} \label{embedding} \iota^{S'}_{S}:X_S \hookrightarrow X_{S'}=X_S \otimes X_{v'}, \quad x \mapsto x \otimes x_{v'},
\end{equation}
which is an isometric embedding.
Using these embeddings, we turn the collection $X_S$, where $S$ runs over finite sets containing $I_0$, into an ascending chain and form
\begin{equation}
\label{Xnaught}
X_0 :=\bigcup_S X_S.
\end{equation}
Note that $X_0$ is spanned by elements $w=\otimes_v w_v$ such that for all but finitely many $v$, we have $w_v$ is the distinguished vector $x_v$. Given such an element $w=\otimes_v w_v \in X_0$ with all but finitely many $w_v$ equals the distinguished vectors $x_v$, and $a=\otimes_v a_v\in A_{0}$ we define $w{\cdot}a$ via the rule
$$w{\cdot}a:= \otimes_{v}(w_v {\cdot} a_v)$$
that is, we act component-wise. It holds that $w_v {\cdot} a_v=w_v=x_v$, for all but finitely many $v$, because $(x_v)$ is compatible with $(p_v)$, so $w{\cdot}a$ lies in $X_0$. We can extend this operation by linearity to a right action of $A_{0}$ on $X_0$.

We will now describe the pre-Hilbert $A_{0}$-module structure on $X_0$. For elements $w=\otimes_v w_v, z=\otimes_v z_v\in X_0$, we put
\begin{equation} \label{inner-prod-def} \langle w,z \rangle := \otimes_{v}\langle w_v, z_v \rangle_v.
\end{equation}
Then for all but finitely many places we have 
$$\langle w_v, z_v \rangle_v=\langle x_v, x_v \rangle_v=p_v.$$ 
Hence the inner product $\langle w,z \rangle$ lands in $A_0$. We can extend this inner product by sesquilinearity over $A_{0}$ to all of $X_0$. 

Finally, we need to show that for all $z\in X_{0}$ we have $\langle z,z\rangle \geq 0$ in $A$ and that $\langle z,z\rangle=0$ only if $z =0$. To this end, notice the natural inclusions 
\begin{equation}
\label{eq:inclusions}
\iota^X_S:X_S\hookrightarrow X_0 , \qquad \iota^A_S:A_S\hookrightarrow A_{0},\end{equation} 
where $\iota^X_S$ is $A_S$-linear and inner-product preserving in the sense that for $w,z\in X_S$, we have
\begin{equation}
\label{eq:isometric}
\langle \iota^X_S(w),\iota^X_S(z)\rangle_{X_0}=\iota^A_S\left(\langle w,z\rangle_{X_S}\right).
\end{equation}

Each $X_S$ is a finite exterior tensor product and thus a Hilbert $C^*$-module, and for $z\in X_{S}$, it follows from (\ref{eq:isometric}) that
$$\langle \iota^{X}_{S}(z),\iota^{X}_{S}(z)\rangle=\iota^A_S(\langle z,z\rangle_{X_S})\geq 0,\quad 
\langle \iota^{X}_{S}(z),\iota^{X}_{S}(z)\rangle=0 \Rightarrow \iota^{A}_{S}(\langle z,z\rangle_{X_S}) =0,
$$
implying that Equation \eqref{inner-prod-def} defines an positive and non-degenerate inner product on $X_0$.

We have shown that $X_0$ is a right pre-Hilbert $C^*$-module over $A_{0}$. It is now a standard procedure to take completions and obtain a $C^*$-module over $A$. The above discussion can be summarized into the following definition. 

\begin{definition}
\label{main-result-1} Let $(X_{v},x_{v})_{v\in I}$be a compatible collection of right Hilbert $C^{*}$-modules over the collection $(A_{v},p_{v})_{v\in I}$ of $C^*$-algebras as in Definition \ref{def:compvector}. By completing the pre-Hilbert $C^*$-module $X_0$ (see \ref{Xnaught}) with respect to the norm arising from the inner product (\ref{inner-prod-def}), we obtain a right Hilbert $C^*$-module, denoted
$$\sideset{}{'}{\bigotimes}_{v \in I} (X_v,x_v),$$
over the restricted product $C^*$-algebra
$$\sideset{}{'}{\bigotimes}_{v \in I} (A_v,p_v).$$ 
We call this the {\it restricted tensor product} of the collection $(X_{v},x_{v})_{v\in I}$.
\end{definition}

\subsection{Direct limit construction}

A robust construction of $\sideset{}{'}{\bigotimes}_{v \in I} (X_v,x_v)$ can be achieved by promoting $(X_S)_S$ to a directed system of Hilbert $C^{*}$-modules over $A$ and taking the direct limit in this category. For details on direct limits of Hilbert $C^{*}$-modules, see \cite[Section 2]{BergmannConti}.

Observe that the space $X_S$ is a finite exterior tensor product that forms a right Hilbert $A_S$-module in the obvious way. 
We have the following.

\begin{proposition}
\label{propconnmap}
For a finite set $S$ containing $I_0$ and $v\notin S$, the map 
$$X_S\otimes_{A_S} A_{S \sqcup \{ v\}} \to X_{S \sqcup \{ v\}}, \quad x\otimes a_v\mapsto x\otimes x_va_v,$$
defined using that $X_S\otimes_{A_S} A_{S \sqcup \{ v\}}=X_S\otimes A_v$ and $X_{S \sqcup \{ v\}}=X_S\otimes X_v$, is $A_{S \sqcup \{ v\}}$-linear, adjointable, and induces an isometric map of $A_{S \sqcup \{ v\}}$-modules
$$X_S\otimes p_vA_v\to X_{S \sqcup \{ v\}}.$$
In particular, the system 
$$\{ X_S\otimes_{A_S} A \}_{S},$$
where $S$ are as above, forms a directed system of right $A$-Hilbert $C^*$-modules with all connecting maps being partial isometries.
\end{proposition}

We now show that the direct limit of $\{ X_S\otimes_{A_S} A \}_{S}$ captures the restricted tensor product $\sideset{}{'}{\bigotimes}_{v \in I} (X_v,x_v)$ of Definition \ref{main-result-1}.

\begin{theorem} 
\label{indnlimit}
We have 
$$\sideset{}{'}{\bigotimes}_{v \in I} (X_v,x_v) \simeq  \varinjlim_S X_S \otimes_{A_S} A$$
 as right $A$-Hilbert $C^*$-modules.
\end{theorem}

\begin{proof}
Recall that $S$ ranges over  finite subsets of $I$ containing $I_0$. The proposition follows if we can produce a bounded $A_{0}$-linear inner product preserving map 
$$\Phi:\varinjlim_S^{\rm alg} X_S\otimes_{A_S}A_{0} \to X ,$$
with dense range. The superscript indicates we take the direct limit in the algebraic category of modules for the ring $A_{0}$. First of all observe that we may take the direct limit over all finite sets $S$ containing $I_0$. On simple tensors we define 
$$\Phi_{S}: X_{S}\otimes_{A_{S}}^{\rm alg}A_{0}\to X,\quad \Phi_S\left(\left(\otimes_{v\in S} w_v\right)\otimes_{A_S} a\right):=wa.$$
We note that $\Phi_S$ is well defined and it is inner product preserving after ``fixing the tail'' in the following sense. We can write $A_0=A_S\otimes \tilde{A}_{0,S}$ where $\tilde{A}_{0,S}:=\bigcup_{S': S'\cap S=\emptyset} A_{S'}$. For $S'$ with $S'\cap S=\emptyset$, in particular $S'\cap I_0=\emptyset$, we form $p_{S'}=\otimes_{v\in S'}p_v$. We can define the multiplier $\tilde{p}_S$ of $\tilde{A}_{0,S}$ as $\tilde{p}_S(a)=p_{S'}a$ for $a\in A_{S'}$. Then by Proposition \ref{propconnmap}, each $\Phi_{S}$ is a partial isometry inducing an  inner product preserving map 
$$X_{S}\otimes^{\rm alg}\tilde{p}_S\tilde{A}_{0,S}\to X.$$
We can write $X_S\otimes_{A_S}A_{0}=X_{S}\otimes^{\rm alg}\tilde{A}_{0,S}$, and since the tails in $\tilde{A}_{0,S}$ eventually become $p_v$, we conclude that $\varinjlim_S^{\rm alg} X_S\otimes_{A_S}A_{0}=\varinjlim_S^{\rm alg} X_{S}\otimes^{\rm alg}\tilde{p}_S\tilde{A}_{0,S}$.

Since $\varinjlim_S^{\rm alg} X_S\otimes_{A_S}A_{0}=\varinjlim_S^{\rm alg} X_{S}\otimes^{\rm alg}\tilde{p}_S\tilde{A}_{0,S}$, and the map $\Phi_{S}$ is compatible with the directed system $\{X_{S}\otimes_{A_{S}}A_{0}\}$,  the universal property of direct limits gives an inner product preserving map
\[\Phi:\varinjlim_{S}X_{S}\otimes_{A_{S}}^{\rm alg}A_{0}\to X.\]
The range of $\Phi$ is $X_0{\cdot} A_0$ which is dense in  $\sideset{}{'}{\bigotimes}_{v \in I} (X_v,x_v)$, and the result follows. 
\end{proof}

\subsection{Induction commutes with restricted tensor product}
The following relates the construction above to how local induction functors glue together to give a global one.

\begin{proposition}  
\label{local-global-compatibility-1}
Let $(X_v,x_{v})_{v\in I}$ be a compatible collection of Hilbert $C^{*}$-modules over $(A_{v},p_{v})_{v\in I}$, and $X:=\sideset{}{'}{\bigotimes}_{v \in I} (X_v,x_v)$ their restricted tensor product. If $$(\pi, \mathcal{H}) \simeq \left ( \sideset{}{'}\bigotimes \pi_v, \sideset{}{'}\bigotimes_{v}( \mathcal{H}_v,h_v) \right )$$ is  a representation of 
$A$, then we have a canonical, unitary isomorphism
$$X \otimes_{A} \mathcal{H} \simeq \sideset{}{'}\bigotimes \left( X_v \otimes_{A_v} \mathcal{H}_v ,x_v\otimes_{A_v}h_v\right ),$$
defined on simple tensors by 
$$
(\otimes_{v} w_v)\otimes (\otimes_{v} y_v)\mapsto \otimes_{v} (w_v\otimes y_v).
$$
\end{proposition}

We return to a more general result than Proposition \ref{local-global-compatibility-1} below in Proposition \ref{IP-commutes-RTP}, and the former will follow from the latter.

\subsection{Compact operators on restricted tensor products}

We now turn to describing the compact operators on a restricted tensor product of Hilbert $C^*$-modules.

\begin{proposition}
\label{comaokdapdops}
Let $(X_v,x_{v})_{v\in I}$ be a compatible collection of Hilbert $C^{*}$-modules over $(A_{v},p_{v})_{v\in I}$, and $X:=\sideset{}{'}{\bigotimes}_{v \in I} (X_v,x_v)$ their restricted tensor product. The inner-product preserving inclusions $\iota^{X}_{S}:X_S\to X$ from Equation \eqref{eq:inclusions} induce an isomorphism of $C^*$-algebras 
$$\mathbb{K}_A(X)\simeq \sideset{}{'}{\bigotimes}_{v \in I} (\mathbb{K}_{A_v}(X_v),p_{x_v}).$$
Here $p_{x_v}$ are the projections defined in (\ref{pxv}).
\end{proposition}
\begin{proof}
Let $S$ be a finite set containing $I_0$ and $S'=S\sqcup\{v\}$. The inner-product preserving inclusions $\iota^{X}_{S}:X_S\to X$ and $\iota_{S}^{S'}:X_{S}\to X_{S'}$ from Equations \eqref{embedding} and \eqref{eq:inclusions} induce isometric $*$-homomorphisms
\[\alpha^{S'}_{S}:\mathbb{K}(X_{S})\to \mathbb{K}(X_{S'}),\quad \alpha^{X}_{S}:\mathbb{K}(X_{S})\to\mathbb{K}(X),\]
defined on rank one operators by $\alpha^{S'}_{S}(T_{\xi,\eta}):=T_{\iota^{S'}_{S}(\xi),\iota^{S'}_{S}(\eta)}$, and similarly for $\alpha^{X}_{S}$. 
We have commutative diagrams
\[\xymatrix{X_{S}\ar[d]_{\iota^{S'}_{S}}\ar[r]^{\iota^{X}_{S}} & X & & \mathbb{K}_{A_{S}}(X_{S})\ar[d]_{\alpha^{S'}_{S}}\ar[r]^{\alpha^{X}_{S}} & \mathbb{K}_{A}(X) \\
X_{S'}\ar[ru]_{\iota^{X}_{S'}} & & &  \mathbb{K}_{A_{S'}}(X_{S'})\ar[ru]_{\alpha^{X}_{S'}}. &}
\]
A short computation shows that for $T\in\mathbb{K}(X_{S})$ we have $\alpha_{S}^{S'}(T)=T\otimes p_{x_{v}}$, so the direct limit of the $\mathbb{K}(X_{S})$ along the maps $\alpha_{S}^{S'}$ coincides with $\sideset{}{'}{\bigotimes}_{v \in I} (\mathbb{K}_{A_v}(X_v),p_{x_v})$, and we obtain an isometric $*$-homomorphism
\begin{equation} \label{isometric-embedding} \sideset{}{'}{\bigotimes}_{v \in I} (\mathbb{K}_{A_v}(X_v),p_{x_v})\to \mathbb{K}(X).
\end{equation}

Since $X_0$ is dense in $X$, the span of the image under $\alpha^{X}_{S}$  of the operators $T_{\xi,\eta} \in \mathbb{K}_{A_{S}}(X_{S})$ is dense in $\mathbb{K}_A(X)$. It follows that the image of the embedding (\ref{isometric-embedding}) is dense, finishing the proof.
\end{proof}

\section{Restricted tensor products of $C^*$-correspondences} 
\label{main-resultscorr}

The application of Hilbert $C^*$-modules we are interested in is hinted at in Proposition \ref{local-global-compatibility-1}, but what is missing is how the Hilbert $C^*$-module transfers a representation of $A$ to another $C^*$-algebra $A'$. The appropriate tool for such a construction is a $C^*$-correspondence. Recall that an $(A',A)$-correspondence is an $A$-Hilbert $C^*$-module $X$ equipped with a left action of $A'$ defined in terms of a $*$-homomorphism
$$\alpha:A'\to \End_A^*(X).$$
We will in this section define restricted tensor products of $C^*$-correspondences. In \cite{BergmannConti} the infinite internal tensor product of a family of $(A,A)$-correspondences was defined as an $(A,A)$-correspondence. The construction here will be based on the external tensor product, and will produce a $C^{*}$-correspondence over the restricted tensor products of the coefficients. 

We remain in the set-up of Section \ref{main-results}. That is $(A_v,p_v)_{v \in I}$ will be a collection of $C^*$-algebras equipped with distinguished projections, $(X_v,x_{v})_{v \in I}$ is a compatible collection of Hilbert $C^{*}$-modules over $(A_{v},p_{v})_{v\in I}$ as in Definition \ref{def:compvector}. We now also introduce a second collection of $C^*$-algebras equipped with distinguished projections $ (A'_v,p'_v)_{v \in I}$ as in Subsection \ref{basic-family}. 

\begin{definition} 
\label{asstwo}
We say that a collection $(\alpha_v)_{v \in I}$ of $*$-homomorphisms
\begin{equation} 
\label{local-left-action} 
\alpha_v : A'_v \longrightarrow {\rm End}_{A_v}^*(X_v),
\end{equation}
is {\em compatible with} $(X_{v},x_v)_{v \in I}$ if for all but finitely many places it holds that
$$\alpha_v(p'_v){\cdot}x_v=x_v.$$ 
If this is the case, we also say that $(X_{v},x_{v})_{v \in I}$ is a collection of $\left ( (A'_v,p'_{v}),(A_v,p_{v}) \right )_{v \in I}$-correspondences.
\end{definition}

As before, let us put 
$$A':=\sideset{}{'}{\bigotimes}_{ v \in I} (A'_v,p'_v). \qquad A:=\sideset{}{'}{\bigotimes}_{v \in I} (A_v,p_v).$$
Let
$$X:=\sideset{}{'}{\bigotimes}_{v \in I} (X_v,x_v)$$
denote the restricted product of the family $(X_v,x_v)_{v \in I}$. 

Consider  the map
$$\alpha: A' \longrightarrow {\rm End}_{A}^*(X)$$
defined by declaring the action for simple tensors $w=\otimes_{v}w_v \in X$ and  $a'=\otimes_{v}a'_v \in A'$ to be
$$\alpha(a')(w)= \otimes_{v}(\alpha_v(a'_v)(w_v)).$$
It can be shown that $\alpha$ is a $*$-homomorphism by giving it a functorial definition. There is an obvious left action $\alpha_S: A'_S \longrightarrow {\rm End}^*_{A_S}(X_S)$. The connecting maps  $\iota^{S'}_{S}$ defined in Equation \eqref{embedding} are clearly compatible with $(\alpha_S)_{S \, finite}$, so there is an induced map $\alpha: A' \longrightarrow {\rm End}^*_{A}(X)$ in the limit of the maps $A'_S \longrightarrow {\rm End}^*_{A}(X)$ induced by $\alpha_S$ and the natural maps ${\rm End}^*_{A_S}(X_S)\longrightarrow {\rm End}^*_{A}(X)$.

We summarize the above as follows.

\begin{definition} 
\label{main-result-2}
 Let $I$ be a countable index set. Assume that for every $v \in I$, we are given a tuple $(X_v,x_v, A'_v,p'_v,A_v,p_v)$ where 
$X_v$ is an $(A'_v,A_v)$-correspondence, $x_v$ is a distinguished vector in $X_v$ and $p'_v,p_{v}$ are projections in $A'_v$ and $A_v$ respectively. Assume  further that for all but finitely many $v \in I$, we have that $x_v, p'_v,p_v$ are nonzero and that 
$$\langle x_v, x_v \rangle\subrangle{v}=p_v,$$
and 
$$p'_v{\cdot}x_v=x_v.$$ 
Then as discussed above, the restricted tensor product 
$$\sideset{}{'}{\bigotimes}_{v \in I} (X_v,x_v)$$ 
carries a natural $\left ( \sideset{}{'}{\bigotimes}_{ v \in I} (A'_v,p'_v),\sideset{}{'}{\bigotimes}_{ v \in I} (A_v,p_v) \right )$-correspondence structure, which we call the {\it restricted tensor product} of the collection 
$(X_v,x_v, A'_v,p'_v,A_v,p_v)_{v \in I}$ of $C^*$-correspondences.
\end{definition}

\subsection{Internal product commutes with restricted tensor product}

\begin{proposition} 
\label{IP-commutes-RTP} 
Let $(X_v,x_{v})_{v \in I}$ and  $(Y_v,y_{v})_{v \in I}$ be collections of $((A_{v}',p_{v}'), (A_v,p_v))_{v \in I}$ and $((A_{v},p_{v}), (A^{''}_v,p^{''}_v))_{v \in I}$ correspondences respectively, both as in Definition \ref{asstwo}.  Let us put 
$$A':= \sideset{}{'}\bigotimes_{v \in I} (A'_v,p'_v), \quad A:= \sideset{}{'}\bigotimes_{v \in I} (A_v,p_v), \quad A'':= \sideset{}{'}\bigotimes_{v \in I} (A^{''}_v,p^{''}_v),$$
Then  then we have a canonical, unitary isomorphism of  $\left (A',A'' \right )$-correspondences
$$\left ( \sideset{}{'}\bigotimes_{v \in I} (X_v,x_v) \right ) \bigotimes_{A} \left ( \sideset{}{'}\bigotimes_{v \in I} (Y_v,y_v) \right )  \simeq \sideset{}{'}\bigotimes_{v \in I} \left( X_v \otimes_{A_v} Y_v , x_v\otimes_{A_v} y_v\right ),$$
defined on simple tensors by 
\begin{equation}
\label{canonon}
(\otimes_{v} w_v)\otimes (\otimes_{v} y_v)\mapsto \otimes_{v} (w_v\otimes_{A_v} y_v).
\end{equation}

\end{proposition}

\begin{proof} 
Let us start by showing that $( X_v \otimes_{A_v} Y_v , x_v\otimes_{A_v} y_v)$ is a collection of $((A_{v}',p_{v}'), (A_v'',p_v''))_{v \in I}$-correspondences. It is clear that $p_v'(x_v\otimes_{A_v} y_v)=x_v\otimes_{A_v} y_v$ and, moreover,
$$\langle x_v\otimes_{A_v} y_v,x_v\otimes_{A_v} y_v\rangle_{X_v\otimes_{A_v}Y_v}=\langle y_v,p_vy_v\rangle_{Y_v}=\langle y_v,y_v\rangle_{Y_v}=p_v''.$$

For notational simplicity, set $X:=\sideset{}{'}{\bigotimes}_{v \in I} (X_v,x_v)$ and $Y:=\sideset{}{'}{\bigotimes}_{v \in I} (Y_v,y_v)$. By linearity, the map \eqref{canonon} extends to a surjection from a dense subspace of $X \otimes_{A} Y$ to a dense subspace of $\sideset{}{'}\bigotimes \left( X_v \otimes_{A_v} Y_v,x_v\otimes_{A_v}y_v \right )$. A short computation shows that on this dense subspace the map \eqref{canonon} is isometric. Therefore, \eqref{canonon} extends to an isometry with dense range, i.e. a unitary isomorphism.
\end{proof}

Observe that Proposition \ref{local-global-compatibility-1} follows from this proposition once we take $A^{''}_v = \C$, forcing the correspondences $Y_v$ to be simply $*$-representations of $A_v$'s.

\subsection{Compact action on the left}
In many applications the left action of $A'_v$ is by compact Hilbert $C^*$-module operators. It is of interest to know whether these compact actions bundle up to a left $A'$-action by $A$-compact operators. 
 
\begin{definition} 
\label{assthree}
Assume that $(X_{v},x_{v})_{v\in I}$ is a collection of $((A'_v,p'_v),(A_v,p_v))_{v\in I}$ correspondences as in Definition \ref{asstwo}. 

\begin{itemize}
\item We say that $(X_{v},x_v)_{v\in I}$ is a {\em coherent collection of correspondences} if at all but finitely many places, 
$$\alpha_v(p'_v)=p_{x_v},$$
where $p_{x_v}\in \mathbb{K}_{A_v}(X_v)$ is the projection along $x_v \in X_v$ (cf. Proposition \ref{comaokdapdops}).
\item A coherent collection of correspondences $(X_{v},x_v)_{v\in I}$ is {\em compact} if also the range of the map $\alpha_v$ is contained in $\mathbb{K}_{A_v}(X_v)$ at all places $v$.
\item A coherent collection of correspondences $(X_{v},x_v)_{v\in I}$ is {\em type I} if also the range of the map $\alpha_v$ contains $\mathbb{K}_{A_v}(X_v)$ at all places $v$.
\end{itemize}
\end{definition}

\begin{remark}
We note that $(X_v,x_v)_{v\in I}$ is a coherent collection of $((A'_v,p'_v),(A_v,p_v))_{v\in I}$-correspondences if and only if for all but finitely many places 
$$\alpha_v(p'_v)X_v=x_vA_v.$$
This follows from the compatibility condition of Definition \ref{def:compvector}.
\end{remark}

\begin{proposition}
\label{kljnajkdna}
When $(X_v,x_v)_{v\in I}$ is a coherent collection of $((A'_v,p'_v),(A_v,p_v))_{v\in I}$-correspondences, the map $\alpha$ from Definition \ref{main-result-2} satisfies the range condition that 
$$\alpha(A') \subseteq \sideset{}{'}{\bigotimes}_{v \in I} (\End_{A_v}^*(X_v),p_{x_v}).$$ 
Moreover, if $(X_v,x_v)_{v\in I}$ is compact, the map $\alpha$ from Definition \ref{main-result-2} satisfies the range condition that 
$$\alpha(A')\subseteq \mathbb{K}_{A}(X),$$ 
that is, $A'$ acts by $A$-compact operators on $X$. And finally, if $(X_v,x_v)_{v\in I}$ is type I, the map $\alpha$ from Definition \ref{main-result-2} satisfies the range condition that 
$$\alpha(A') \supseteq \mathbb{K}_{A}(X),$$ 
that is, the image of $A'$ contains all $A$-compact operators on $X$.
\end{proposition}

\begin{proof}
Let $\iota_S:X_S\to X$ \ be the isometry in Equation \eqref{eq:inclusions}. Because we have a coherent collection of correspondences, $\alpha_v(p'_v)=p_{x_v}$ for all but finitely many indices, and for any finite $S$ the following diagram commutes:
\begin{equation}
\label{ljnkjlnkjnad}
\begin{tikzcd}
A'_S \ar[rr,"\alpha_S"]\ar[dd,"\alpha"]& &\End_{A_S}^*(X_S)\ar[dd]  \ar[ddll,"\mathrm{Ad}(\iota_{S})"] \\
&&\\
 \End_{A}^*(X)& &\sideset{}{'}{\bigotimes}_{v \in I} (\End_{A_v}^*(X_v),p_{x_v})\ar[ll]  
\end{tikzcd}
\end{equation}
Here we use the functorial maps 
$$\End_{A_S}^*(X_S)=\bigotimes_{v\in S} \End_{A_v}^*(X_v)\hookrightarrow \sideset{}{'}{\bigotimes}_{v \in I} (\End_{A_v}^*(X_v),p_{x_v}),$$ 
in the vertical right arrow and the natural inclusion $\sideset{}{'}{\bigotimes}_{v \in I} (\End_{A_v}^*(X_v),p_{x_v})\hookrightarrow  \End_{A}^*(X)$ in the bottom horizontal arrow. By taking the limit in the commuting diagram \eqref{ljnkjlnkjnad} we conclude that $\alpha(A') \subseteq \sideset{}{'}{\bigotimes}_{v \in I} (\End_{A_v}^*(X_v),p_{x_v})$.

If $(X_{v},x_v)_{v\in I}$ is a compact collection of $(A_v',p'_v),(A_v,p_v))_{v\in I}$-correspondences, then the algebra $A'_S$ \ acts as $A_S$-compact operators on $X_S$. Since $A_S \subseteq A$ is a subalgebra, we conclude that as soon as $\alpha_v : A'_v \longrightarrow \mathbb{K}_{A_v}(X_v)$ for all places $v$, we have the following range condition for any finite set of places $S$:
\begin{align*}
\alpha_S: A'_S &\longrightarrow \mathbb{K}_{A}(X_S).
\end{align*}
Since $\iota_{S}$ is inner-product preserving, we have a $*$-homomorphism
$$\alpha^{X}_{S}:\mathbb{K}_{A_{S}}(X_S) \to \mathbb{K}_{A}(X).$$
We can then combine these two facts with Proposition \ref{comaokdapdops}, into an argument similar to above showing that if we have a compact correspondence, then $\alpha(A')\subseteq \mathbb{K}_{A}(X)$. 

The proof that for a type I collection we have $\alpha(A')\supseteq \mathbb{K}_{A}(X)$ can with Proposition \ref{comaokdapdops} be reduced to a straightforward density argument.
\end{proof}

\begin{remark}
The condition of coherence in Definition \ref{assthree} is necessary in compact compatibility for the conclusion of Proposition \ref{kljnajkdna} to hold. The example $I=\mathbb{N}$, $A_v'=M_2(\C)$, $p_v'=1$, $A_v=\C$, $p_v=1$, $X_v=\C^2$ and $x_v=(1,0)^T$ provides an example where the range of the map $\alpha_v$ is contained in $\mathbb{K}_{A_v}(X_v)$ at all places $v$ holds but we do not have a coherent collection of correspondences and the conclusion of Proposition \ref{kljnajkdna} fails to hold since in this case $A'$ is unital and $X$ is an infinite-dimensional Hilbert space. 
\end{remark}

\subsection{A digression on coherent collections of correspondences}

We discuss how the condition of coherence in a collection of correspondences in Definition \ref{assthree} holds under some hypotheses that hold whenever $(X_v,x_v)_{v\in I}$ is a collection of modules with each $X_v$ defining a Morita equivalences from an ideal in $A_v'$ to an ideal in $A_v$. This situation is of interest in the companion paper \cite{GMSTheta} on global $\theta$-correspondences. 

\begin{proposition} 
\label{autoassthree}
Let $(X_v,x_v)_{v\in I}$ be a collection of $((A_v',p_v'),(A_v,p_v))_{v\in I}$-correspondences. If for all but finitely many indices $v$, we have that
\begin{enumerate} 
\item the balanced tensor product with $X_v$, i.e. $(\pi,\mathcal{H})\mapsto (\alpha_v\otimes _\pi 1, X_v\otimes _{A_v}\mathcal{H})$, takes irreducible representations of $A_v$ to irreducible representations of $A_v'$,
\item the projections $p'_v \in A'_v$ are of rank at most one, and 
\item if the balanced tensor product with $X_v$ takes an irreducible representation $\pi_v$ of $A_v$ to a representation $\pi_v'$ of $A_v'$ such that $\pi_v'(p_v')\neq 0$, then $\pi_v(p_v)\neq 0$,
\end{enumerate} 
then $(X_v,x_v)_{v\in I}$ is coherent collection of correspondences. 
\end{proposition}

\begin{proof} For a representation $(\pi,H_{\pi})$ of $A_{v}$ we write $Q_{\pi}:=p_{v}'\otimes 1_{H_{\pi}}$ and $P_{\pi}:=p_{x_{v}}\otimes 1_{H_{\pi}}$. We note that for $T\in \End^{*}_{A_{v}}(X_{v})$, $T\otimes 1_{H_{\pi}}=0$ if and only if for all $x\in X_v$ it holds that $\pi(\langle Tx,Tx\rangle_{X_v})=0$. Since the irreducible representations separate the points of $A_{v}$, in order to show that $p_{v}'=p_{x_{v}}$ it suffices to show that  have $Q_{\pi}=P_{\pi}$ for all irreducible representations $(\pi,H_{\pi})$.

By Definition \ref{asstwo}, the projections $p_{v}',p_{x_{v}}\in\End^{*}_{A_{v}}(X_{v})$ satisfy $p_{v}'p_{x_{v}}=p_{x_{v}}$. For an irreducible representation $(\pi,H_{\pi})$ we thus have $Q_{\pi}P_{\pi}=P_{\pi}$ and $\rm{im}\,P_{\pi}\subset \rm{im}\,Q_{\pi}$. Therefore $Q_{\pi}=0$ implies $P_{\pi}=0$ and in particular then $P_{\pi}=Q_{\pi}$. When $Q_{\pi}\neq 0$, properties (1) and (2) give that $Q_{\pi}$ has rank one. Moreover, property (3) implies that also $P_{\pi}\neq 0$, and since $\rm{im}\,P_{\pi}\subset \rm{im}\,Q_{\pi}$, we must have $P_{\pi}=Q_{\pi}$ as well.
\end{proof}

\subsection{Type I representations and type I collections}

We discuss the reason for the term type I collection in Definition \ref{assthree}. We say that an irreducible representation $(\pi,H)$ of a $C^*$-algebra is type $I$ if $\pi(A)\supseteq \mathbb{K}(H)$. Note that by Guichardet's theorem (see Theorem \ref{factorization-Cstar}) if $A=\sideset{}{'}{\bigotimes}_{v \in I} (A_v,p_v) $, for $(A_v,p_v)_{v \in I}$ a collection of type I $C^*$-algebras equipped with distinguished projections of rank at most one as in Subsection \ref{basic-family}, then any irreducible representation $(\pi,H)$ factors as
\[H=\sideset{}{'}{\bigotimes}_{v \in I} (H_v,h_v),\quad \pi=\sideset{}{'}{\bigotimes}_{v \in I} \pi_v.\] Proposition \ref{comaokdapdops} implies that $\pi$ is always type I, given that each $A_v$ is type I.

The definition above extends to correspondences, we say that an $(A',A)$-correspondence $X$ is type $I$ if $A'$ acts on $X$ via a $*$-homomorphism $\alpha:A'\to \End_{A}^*(X)$ with $\alpha(A')\supseteq \mathbb{K}_A(X)$. We make the following observation. 

\begin{proposition} 
\label{kjfkjnadak}
Assume that $X$ is an $(A',A)$-correspondence and $Y$ an $(A,A'')$-correspondence. If $X$ and $Y$ are type $I$, then $X\otimes_AY$ is also type $I$. 
\end{proposition}

We note in particular that if $X$ is a type $I$ $(A',A)$-correspondence and $(\mathcal{H},\pi)$ is a type $I$ representation of $A$, then $X\otimes\mathcal{H}$ is a type $I$ representation of $A'$.

\begin{proof}
We identify $ \mathbb{K}_A(X)\otimes 1_Y$ with a subalgebra of $\End_{A''}(X\otimes_A Y)$. If $Y$ is type $I$, we have the inclusion
$$\mathbb{K}_{A''}(X\otimes Y)\subseteq \mathbb{K}_A(X)\otimes 1_{Y}.$$
Write $\alpha$ for the left action of $A'$ on $X$. If $X$ is type $I$, then $\mathbb{K}_A(X)\subseteq \alpha(A')$ and we have the inclusions
$$\mathbb{K}_{A''}(X\otimes Y)\subseteq \mathbb{K}_A(X)\otimes 1_{Y}\subseteq \alpha(A')\otimes 1_{Y}=(\alpha\otimes 1)(A').$$
Here $\alpha\otimes 1$ is the left action of $A'$ on $X\otimes_A Y$, and the proof is complete.
\end{proof}

From Proposition \ref{kljnajkdna} and Proposition \ref{kjfkjnadak} we conclude the following. 

\begin{proposition} 
Let $(X_v,x_{v})_{v \in I}$ be a type $I$ collection of $((A_{v}',p_{v}'), (A_v,p_v))_{v \in I}$-correspondences. Then $X:=\sideset{}{'}\bigotimes \left( X_v, x_v\right )$ is also type $I$.

In particular, in this case the induction map from representations of $A$ to representations of $A'$ defined from $X$ as in Definition \ref{main-result-2} satisfies that on factorizable $A$-representations $$(\pi, \mathcal{H}) \simeq \left ( \sideset{}{'}\bigotimes \pi_v, \sideset{}{'}\bigotimes_{v}( \mathcal{H}_v,h_v) \right ),$$ the $A'$-representation
$$
X \otimes_{A} \mathcal{H} \simeq \sideset{}{'}\bigotimes \left( X_v \otimes_{A_v} \mathcal{H}_v ,x_v\otimes_{A_v}h_v\right ),
$$
is type I as soon as $(\pi, \mathcal{H})$ is.
\end{proposition}

\section{$C^*$-correspondences for adelic parabolic induction}
\label{sec:globalpara}

A direct application of the results of Section \ref{main-resultscorr} is to the $C^*$-correspondence approach of Rieffel to induction of unitary representations. As a consequence, we will treat the theory of parabolic induction of representations which was studied via  $C^*$-correspondence in the series of works \cite{Clare-13,CCH-16,CCH-18}.

\subsection{Induced representations}  

We start with a quick recall of the theory of induction of unitary representations (for details, see \cite{book_Folland}) after which we discuss parabolic induction (see more in \cite{getzhahn,knapprep,wallachbook}). Let $G$ be a locally compact Hausdorff group with a closed subgroup $H$. Given a unitary representation $(\pi,V_\pi)$ of $H$, we consider the space $L^G_H(V_\pi)$ of continuous functions $f : G \to V_\pi$ satisfying the conditions 
\begin{enumerate}
\item[(i)] the support of $f$ has compact image under the projection $G \to G/H$,
\item[(ii)]  $f(gh)=\pi(h^{-1})f(g)$ for all  $h \in H$ and  $g \in G$. 
\end{enumerate}
We fix a quasi-invariant Borel measure $\mu$ on $G/H$. 
The Hermitian the form
$$\langle f_1, f_2 \rangle := \int_{G/H} \langle f_1(g), f_2(g) \rangle\subrangle{V_\pi} d\mu(g)$$
equips $L^G_H(V_\pi)$ with a pre-Hilbert space structure. We denote the Hilbert space completion of $L^G_H(V_\pi)$ by ${\rm Ind}_H^G(V_\pi)$. The group $G$ acts on $L^G_H(V_\pi)$ via left translation: 
$$(g{\cdot}f)(g')=f(g^{-1}g').$$
To obtain a unitary representation of $G$ on ${\rm Ind}_H^G(V_\pi)$, we need to tweak the action of $G$ as follows. Let $\delta_G, \delta_H$ denote the modulus characters of $G$ and $H$. The function $h \mapsto \delta_G(h)/\delta_H(h)$ on $H$ admits an extension, denoted $\Delta$, to $G$. One can verify that the action 
$$(g{\cdot}f)(g')=\Delta(g)^{1/2}f(g^{-1}g')$$
defined a unitary action of $G$ on ${\rm Ind}_H^G(V_\pi)$. The equivalence class of the unitary representation ${\rm Ind}_H^G(\pi)$ does not depend on the choice of the quasi-invariant measure $\mu$.

The case when $G$ is a reductive group over the integers and $H=P$ for a parabolic subgroup $P\subseteq G$ is of particular interest. For a local field $F$, assume that $(\pi,V_\pi)$ is a unitary representation of the Levi factor $L(F)$ in the Langlands factorization $P=LN$, and extend $\pi$ trivially to $P(F)$ by $\pi(ln)=\pi(l)$, then ${\rm Ind}_{P(F)}^{G(F)}(V_\pi)$ is called a parabolically induced representation. Parabolic induction plays an important role in parametrizing all admissible representations of reductive groups, see e.g. \cite{BZind, knapprep}.

\subsection{Rieffel's induction $C^*$-correspondence} Let $G$ be a locally compact Hausdorff group with a closed subgroup $H$. One obtains a $(C^*(G),C^*(H))$-correspondence $X^G_H$ by completing $C_c(G)$ in the norm obtained from the $C_c(H)$-valued inner product defined from the matrix coefficients of the natural right action of $H$ on $C_c(G)$. It is useful to note that this Hilbert module structure can also be obtained by viewing $C_c(G)$ as a pre-Hilbert module over itself in the standard way (i.e. by declaring $\langle a,b \rangle$ to be $a^*b$) and then descending to a pre-Hilbert module over $C_c(H)$ by applying the restriction map $p:C_c(G) \to C_c(H)$ (which is a conditional expectation). The natural left action of $G$ on $C_c(G)$ preserves this inner product and extends to an action of $C^*(G)$ on $X^G_H$ by adjointable $C^*(H)$-endomorphisms. 

Let $(\pi,V_\pi)$ be a unitary representation of $H$. Viewing it as a $C^*(H)$-representation, consider the $C^*(G)$-representation on the Hilbert space 
$X^G_H \otimes_{C^*(H)} V_\pi$. This is the integrated form of a unitary $G$-representation which we will denote ${\rm Ind}(X^G_H,\pi)$. The following is a well-known result of Rieffel (see \cite[Theorem 5.12]{Rieffel-74}).

\begin{theorem} Let $(\pi,V_\pi)$ be a unitary representation of $H$. The map
\begin{equation} \label{isometry} \Psi: X^G_H \otimes_{C^*(H)} V_\pi \longrightarrow {\rm Ind}^G_H(V_\pi),
\end{equation} 
defined by
$$\Psi(f \otimes_{C^*(H)} v)(s) := \int_H f(sh) \pi(h)(v) dh, \qquad f \in C_c(H), \ v \in V,$$
is an isometry of Hilbert spaces and it intertwines the $G$-representations ${\rm Ind}(X^G_H,\pi)$ and ${\rm Ind}_H^G(\pi)$. In other words, the $(C^*(G),C^*(H))$-correspondence $X^G_H$ captures the induction functor.
\end{theorem}

\subsection{Local-global compatibility} 

Let $(G_v,K_v)_{v \in I}$ be an admissible family of groups as in Definition \ref{def:adm} and for all $v\in I$, we fix a closed subgroup $H_v\subseteq G_v$. As in  Subsection \ref{setup}, we can form the locally compact groups
$$\mathbb{G}:=\sideset{}{'}{\prod} (G_v : K_v)  \quad\mbox{and}\quad \mathbb{H}:=\sideset{}{'}{\prod} (H_v : K_v\cap H_v)$$
The inclusions $H_v\hookrightarrow G_v$ allow us to view $\mathbb{H}$ as a closed subgroup of $\mathbb{G}$. 

For all the indices $v$ for which $K_v$ is compact open in $G_v$, we let $p'_v$ denote the projection in $C^*(G_v)$ given by the characteristic function of $K_v$ normalized so that $K_v$ has volume one. Similarly we $p_v$ denote the projection in $C^*(H_v)$ given by the characteristic function of $K_v \cap H_v$ normalized so that $K_v \cap H_v$ has volume one. 

We also let $\phi_v$ denote the element of Rieffel's induction $(C^*(G_v),C^*(H_v))$-correspondence $X^{G_v}_{H_v}$ obtained by viewing $p'_v$ as an element of $C_c(G_v)$.  Then it is easy to check that the collection $(X^{G_v}_{H_v}, \phi_v)_{v \in I}$ of  $\left ( (C^*(G_v),p'_v), (C^*(H_v),p_v) \right )_{v \in I}$-correspondences can be assembled into a restricted product as in Definition \ref{main-result-2}. The proposition below states that the $(C^*(\mathbb{G}),C^*(\mathbb{H}))$-correspondence we obtain from this restricted product captures the induction functor for the groups $\mathbb{G}$ and $\mathbb{H}$.

\begin{proposition} \label{induction-local-global} We have an isomorphism
$$X^{\mathbb{G}}_{\mathbb{H}} \simeq \sideset{}{'}\bigotimes_{v} (X^{G_v}_{H_v}, \phi_v)$$
of $(C^*(\mathbb{G}),C^*(\mathbb{H}))$-correspondences.
\end{proposition}
\begin{proof} 
Recall that the Rieffel induction $C^*$-correspondence $X^{\mathbb{G}}_{\mathbb{H}}$ is the completion of $C_{c}(\mathbb{G})$ in the norm coming from the $C_{c}(\mathbb{H})$-valued inner product which arises from the 
restriction map $C_{c}(\mathbb{G}) \to C_{c}(\mathbb{H})$.  Similarly, for each place $v$, $C_{c}(G_{v})$ is a right inner product module over $C_{c}(H_{v})$, which is dense in $X^{G_{v}}_{H_{v}}$. 
For a finite set of indices $S$ containing the infinite places, there are dense inclusions 
$$\mathcal{X}_{S}:=\bigotimes^{\mathrm{alg}}_{v\in S}C_{c}(G_{v})\subseteq X_{S}:=\bigotimes_{v\in S} X^{G_v}_{H_v},$$ 
and 
$$\mathcal{B}_{S}:=\bigotimes^{\mathrm{alg}}_{v\in S}C_{c}(H_{v})\subseteq B_{S}:=\bigotimes_{v\in S} C^{*}(H_{v}).$$ 
If $S'=S \cup \{v\}$, there is a commutative diagram
\begin{equation}
\label{eq:restriction-compatible}
\xymatrix{\mathcal{X}_{S}\ar[r]\ar[d] & \mathcal{B}_{S}\ar[d]\\
\mathcal{X}_{S'}\ar[r] &\mathcal{B}_{S'}, }
\end{equation}
where the horizontal arrows are given by the restriction maps, and the vertical maps are given by $x\mapsto x\otimes \phi_{v}$ and $b\mapsto b\otimes p_{v}$. 


For a finite set $S$ we have a map $\otimes_{v\in S} C_{c}(G_{v})\to C_{c}(\mathbb{G})$ defined by $$(\otimes_{v\in S} f_{v}) (g_{i})_{i\in I}:=\prod_{v\in S} f_{v}(g_{v}),\quad f_{v}\in C_{c}(G_{v}).$$
Thus, we can view $\mathcal{X}_{S}$ as a subspace of $C_{c}(\mathbb{G})$ and view $\mathcal{B}_{S}$ as a subalgebra of $C_c(\mathbb{H})$. Now, since  $C_{c}(\mathbb{G})$ is a module over $C_{c}(\mathbb{H})$, we can form the right $C_{c}(\mathbb{H})$-module maps
\begin{equation}
\label{eq:dense-module-maps}
\mathcal{X}_{S}\otimes^{\mathrm{alg}}_{\mathcal{B}_{S}}C_{c}(\mathbb{H})\to C_{c}(\mathbb{G}),\quad x\otimes f\mapsto x{\cdot} f,
\end{equation}
which by \eqref{eq:restriction-compatible} preserve the inner product. These maps therefore extend to the completions and, for any $S\subseteq S'$ yield a commutative diagram
\[\xymatrix{
X_{S}\otimes_{B_{S}}C^{*}(\mathbb{H})\ar[r]\ar[rd]& X_{S'}\otimes_{B_{S'}}C^{*}(\mathbb{H})\ar[d]\\ &X^{\mathbb{G}}_{\mathbb{H}}}\]
of $C^*(\mathbb{H})$-modules.  By Theorem \ref{indnlimit} and the universal property of direct limits, we obtain an inner product preserving right module map
\[\sideset{}{'}\bigotimes_{v}(X^{G_{v}}_{H_{v}},x_{v}) \simeq \varinjlim_{S} X_{S}\otimes_{B_{S}}C^{*}(\mathbb{H})\to X^{\mathbb{G}}_{\mathbb{H}},\]
which has dense range, since the joint range of the maps \eqref{eq:dense-module-maps} is dense in $X^{\mathbb{G}}_{\mathbb{H}}$. This completes the proof. 
\end{proof}

Now let $(\pi_v, V_v)_v$ be a family of unitary, irreducible representations of $H_v$'s, endowed with $K_v \cap H_v$-invariant vectors $w_v$ for all but finitely many $v$. 
We introduce the element 
\begin{equation} \label{rieffel_chi} \chi_v:=\Psi(\phi_v \otimes_{C^*(H_v)} w) \in L^{G_v}_{H_v}(V_v)
\end{equation}
with $\phi_v$ the distinguished vector in $X^{G_v}_{H_v}$ defined above and $\Psi$ the isometry given in (\ref{isometry}). We have the following immediate corollary of the previous proposition.
\begin{corollary} 
\label{inductionfactor}
We have  
$${\rm Ind}_{\mathbb{H}}^{\mathbb{G}} \left (  \sideset{}{'}{\bigotimes}_{v\in I} \left ( \pi_v, w_v \right ) \right ) \simeq  
\sideset{}{'}{\bigotimes}_{v\in I} \left ( {\rm Ind}_{H_v}^{G_v}(\pi_v), \chi_v \right )$$
as unitary representations of $\mathbb{G}$. 
\end{corollary}
\begin{proof} The claim follows immediately from Proposition \ref{induction-local-global} and Proposition \ref{local-global-compatibility-1}.
\end{proof}

\subsection{Parabolic induction}
Now let $G$ be a reductive linear algebraic group over the integers and $F$ a number field. We introduce subgroups $K_v \subset G(F_v)$ as in Section \ref{adelic_groups}. Fix a parabolic subgroup $P=LN$ of $G$ that we assume is also defined over $F$. Fix a place $v$ of $F$. By an abuse of notation we write $G_v=G(F_v)$, and so on. Note that the parabolic subgroups $P_v\subseteq G_v$ and their Langlands decompositions $P_v=L_vN_v$ are over $F_v$. 

As mentioned above, (local) parabolic induction functor takes as input a unitary representation of $L_v$, extends it trivially to $P_v$ and then applies induction to obtain a unitary representation of $G_v$. In the setting of real reductive groups (that is, for $v$ archimedean), parabolic induction functor was implemented  in \cite{Pierrot-01}  by the $(C^*(G_v),C^*(L_v))$-correspondence 
\begin{equation} \label{pierrot-def} \mathcal{E}^{G_v}_{L_v} := X^{G_v}_{P_v} \otimes_{C^*(P_v)} C^*(L_v),
\end{equation}
where $X^{G_v}_{P_v}$ is Rieffel's induction $C^*$-correspondence and $C^*(L_v)$ is viewed in the standard way as a $(C^*(P_v),C^*(L_v))$-correspondence via the surjection $C^*(P_v) \to C^*(L_v)$ (recall that $L_v \simeq P_v/N_v$).

In \cite{Clare-13}, the correspondence $\mathcal{E}^{G_v}_{H_v}$ was conveniently captured, again in the case of real reductive groups, as a completion of the space of compactly supported continuous functions on the homogeneous space $G_v/N_v$. This description of $\mathcal{E}^{G_v}_{H_v}$ was refined to a $C^*$-correspondence for the reduced group $C^*$-algebras 
$(C^*_r(G_v),C^*_r(L_v))$ and further utilized in \cite{CCH-16,CCH-18}. 

In our approach to parabolic induction, we will employ (\ref{pierrot-def}) instead of the description given in \cite{Clare-13} since the former applies uniformly to all places and builds on Rieffel's induction $C^*$-correspondence that we treated earlier. 

The next result can be seen as a $C^*$-algebraic version of local-global compatibility of the parabolic induction functor.
\begin{proposition} \label{parakjn} Let the setting be as above. We have
$$\mathcal{E}^{G(\A)}_{L(\A)} \simeq   \sideset{}{'}\bigotimes_{v} (\mathcal{E}^{G_v}_{L_v}, \varepsilon_v)$$
where 
$$\varepsilon_v:= \phi_v \otimes_{C^*(P_v)} p_{L_v\cap K_v}.$$
\end{proposition}   
\begin{proof} We have 
\begin{align*} \mathcal{E}^{G(\A)}_{L(\A)} & \simeq X^{G(\A)}_{P(\A)} \otimes_{C^*(P(\A))} C^*(L(\A)) \\
& \simeq \sideset{}{'}\bigotimes_{v} \left ( X^{G_v}_{H_v} \otimes_{C^*(P_v)} C^*(L_v) \right ) \\
\end{align*}
by Corollary \ref{equitotad}, Proposition \ref{IP-commutes-RTP} and Proposition \ref{induction-local-global}.
\end{proof}

\begin{remark}
Proposition \ref{parakjn} raises some questions and problems:
\begin{enumerate}
\item The problem of obtaining results similar to Proposition \ref{parakjn}, in the vein of Langlands- and Bernstein-Zelevinsky classification, for admissible representations would rely on a framework beyond $C^*$--algebras. 
\item Jacquet functors and Bernstein's second adjoint theorem, can be attacked starting from the search for a left inner product. For the real place \cite{CCH-16}, there is a left inner product if we restrict to the reduced group $C^*$-algebra that was used in \cite{CCH-18} for constructing adjoint functors.
\item The problem of characterizing the unitary dual of $G(\A)$, or already the reduced unitary dual, is by Remark \ref{adjnakdjnadkjn}, equivalent to the characterizing the unitary duals in all the local cases. It is not clear to the authors if elegant work such as \cite{afgoaub,CCH-16} can be emulated and compute the reduced unitary dual directly in the global case of $G(\A)$.
\end{enumerate}
\end{remark}


\end{document}